\documentclass[11pt]{amsart}

\usepackage{amssymb}
\usepackage[utf8]{inputenc}
\usepackage{enumerate}
\usepackage[vmargin=2cm,hmargin=2cm]{geometry}
\usepackage{xcolor}
\usepackage{pgf,tikz}

\newtheorem{theorem}{Theorem}
\newtheorem{lemma}[theorem]{Lemma}
\newtheorem{corollary}[theorem]{Corollary}
\newtheorem{proposition}[theorem]{Proposition}

\theoremstyle{definition}
\newtheorem{definition}[theorem]{Definition}
\newtheorem{example}[theorem]{Example}

\theoremstyle{remark}
\newtheorem{remark}[theorem]{Remark}

\title{Affine semigroups having a unique Betti element}

\author{{P.A.} {Garc\'{\i}a S\'anchez}}
\address{Departamento de \'Algebra, Universidad de Granada, E-18071 Granada, Espa\~na}
\email{pedro@ugr.es}

\author{I. Ojeda}
\address{Departamento de Matem\'{a}ticas, Universidad de Extremadura,  E-06071 Badajoz, Espa\~na}
\email{ojedamc@unex.es}

\author{{J.C.} Rosales}
\address{Departamento de \'Algebra, Universidad de Granada, E-18071 Granada, Espa\~na}
\email{jrosales@ugr.es}

\thanks{The first and third authors are supported by the projects MTM2010-15595 and FQM-343,  FQM-5849, and FEDER funds. The second author is supported by the project MTM2007-64704, National Plan I+D+I and by Junta de Extremadura (FEDER funds)}


\begin{document}

\begin{abstract}
We characterize affine semigroups having one Betti element and we compute some relevant non-unique factorization invariants for these semigroups. As an example, we particularize our description to numerical semigroups.
\end{abstract}

\maketitle 
 
\section*{Introduction}

An \textbf{affine semigroup} is a finitely generated submonoid of $\mathbb N^r$ for some positive integer $r$, where $\mathbb N$ denotes the set of nonnegative integers. Thus any affine semigroup is cancellative ($\mathbf{a}+\mathbf{b}=\mathbf{a}+\mathbf{c}$ implies $\mathbf{b} = \mathbf{c}$), reduced (its only unit is $0,$ the identity element), and torsion free ($k \mathbf{a} = k \mathbf{b}$ for $k$ a positive integer implies that $\mathbf{a} = \mathbf{b}$). In fact the converse is also true by Grillet's theorem: \emph{if $S$ is a finitely generated cancellative, reduced, torsion free monoid, then it is an affine semigroup} (see for instance \cite[Theorem 3.11]{RGS99}).  

Every affine semigroup is finitely presented, and this means, that it can be described by a finite presentation: a finite set of `free' generators and a finite set of relations among them.  The elements whose factorizations are involved in the relations of a minimal presentation are called \textbf{Betti elements} (since from them one can compute the minimal free resolution of the semigroup ring associated to the monoid).

In this manuscript we study affine semigroups having just one Betti element. It turns out that these monoids have a flower shape, since they are the gluing of several \textbf{free monoids} (in the categorical sense,  that is, semigroups that are isomorphic as monoid to $\mathbb{N}^s$ for some $s$), and these monoids, the petals, are glued by a single element, the stem (see Theorem \ref{Th.Main}). The concept of gluing and its basic properties are recalled in Section \ref{sec:unico-betti}. For numerical semigroups we are able to give a characterization in terms of the minimal generating set, which allow us to construct as many examples as we like.

Recently, several manuscripts take advantage of minimal presentations to compute non-unique factorization invariants. We show in the last section how to compute the elasticity, the maximum of the Delta set, the catenary and tame degree, and the $\omega$-invariant of any affine semigroup with a single Betti element. As a consequence we obtain that the last three invariants in this list coincide (see Theorem \ref{Th.Main2}), as for monoids having a generic presentation (see \cite{BGSG}).

\section{Affine semigroups with a single Betti element}\label{sec:unico-betti}

Let $S$ be an affine semigroup minimally generated by $\mathcal{A} := \{\mathbf{a}_1,\ldots, \mathbf{a}_p\} \subset \mathbb{N}^r$, that is to say, $S = \mathbb{N} \mathcal{A} := \mathbb{N} \mathbf{a}_1 + \cdots + \mathbb{N} \mathbf{a}_p$ and no proper subset of $\mathcal{A}$ generates $S.$ Let $A$ denote the $r \times p-$integer matrix whose columns are $\mathbf{a}_1,\ldots, \mathbf{a}_p.$ The monoid map $$\pi_\mathcal{A} : \mathbb{N}^p \longrightarrow S;\ \mathbf{u} = (u_1, \ldots, u_p) \longmapsto A \mathbf{u} = \sum_{i=1}^p u_i \mathbf{a}_i$$ is sometimes known as the \textbf{factorization homomorphism} associated to $\mathcal{A}$. For every $\mathbf{a} \in S$, the set $\mathsf{Z}(\mathbf{a}) := \pi_\mathcal{A}^{-1}(\mathbf{a})$ is the \textbf{set of factorizations} of $\mathbf{a}$.

Let $\equiv_\mathcal{A}$ be the kernel congruence of $\pi_\mathcal{A}$, that is, $\mathbf{u} \equiv_\mathcal{A} \mathbf{v}$ if $\pi_\mathcal{A}(\mathbf{u})=\pi_\mathcal{A}(\mathbf{v})$, or equivalenty, $\mathbf{u}$ and $\mathbf{v}$ are factorizations of the same element in $S$. Notice that $S$ is isomorphic to the monoid ${\mathbb N}^p/\! \equiv_\mathcal{A}$.

Given $\sigma \subseteq {\mathbb N}^p\times {\mathbb N}^p$, the congruence generated by $\sigma$ is the least congruence on $\mathbb{N}^p$ containing $\sigma$. If $\equiv_\mathcal{A}$ is the congruence generated by $\sigma$, then we say that $\sigma$ is a presentation of $S$. R\'edei's theorem (see \cite{Redei}) states that every finitely generated monoid is finitely presented. A presentation for $S$ is a minimal presentation if none of its proper subsets generates $\equiv_\mathcal{A}$. In our setting, all minimal presentations have the same cardinality $\nu(S)$ (see for instance \cite[Corollary 9.5]{RGS99}). It is well known that $\nu(S) \geq p-\mathrm{rank}(A)$ and $S$ is said to be a \textbf{complete intersection} if the equality holds.

Moreover, given a minimal presentation $\sigma$ for $S,$ the set $$\big\{ \pi_\mathcal{A}(\mathbf{u})\ \mid\ (\mathbf{u}, \mathbf{v}) \in \sigma \big\} \subset S$$ does not depend on $\sigma$ (see for instance \cite[Chapter 9]{RGS99}). Following \cite{uniquely}, we will write $\textrm{Betti}(S)$ for the above set and we will say that $\mathbf{b} \in \textrm{Betti}(S)$ is a Betti element of $S.$ Notice that $\textrm{Betti}(S) = \varnothing$ if and only if $S$ is a free affine semigroup. Every free affine semigroup is a complete intersection with $\nu(S)=0$. In particular, free affine semigroups are complete intersections.

For an element $\mathbf{u} = (u_1, \ldots, u_p) \in \mathbb N^p$, we define its support as the set $\textrm{supp}(\mathbf{u}) := \{i \mid u_i \neq 0,\ i \in \{1,\ldots,p\}\}$. Notice that $\mathbf{u} = (u_1, \ldots, u_p)$ and $\mathbf{v} = (v_1, \ldots, v_p) \in \mathbb N^p$ have disjoint support if and only if the usual dot product $\mathbf{u} \cdot \mathbf{v} = \sum_{i=1}^p u_i v_i$ equals zero. We emphasize that a necesary (but not sufficient) condition for a presentation $\sigma$ for $S$ to be minimal is that $\mathbf{u}$ and $\mathbf{v}$ have disjoint support and $\mathrm{gcd}(\mathbf{u}, \mathbf{v}) = 1,$ for every $(\mathbf{u}, \mathbf{v}) \in \sigma$ (see, e.g., \cite[Chapter 9]{RGS99}).

\begin{proposition}\label{prop-unificadora}
Let $\mathbf d\in S$ be such that $\mathrm{Betti}(S)=\{\mathbf d\}$. Assume that $\mathsf Z(\mathbf d)=\{\mathbf v_1,\ldots, \mathbf v_s\}$. Let $\mathbf a\in S$.
\begin{itemize}
\item[(a)] $\#\mathsf Z(\mathbf a)>1$ if and only if $\mathbf a-\mathbf d\in S$.

\item[(b)] For every $i\neq j$, $\mathbf v_i \cdot \mathbf v_j=0$. 

\item[(c)] There exists $a\in \mathbb N$, $\mathbf b\in S$ and $\mathbf w\in \mathbb N^p$ with $\mathsf Z(\mathbf b)=\{\mathbf w\}$, such that for every $\mathbf u\in \mathsf Z(\mathbf a)$,  $\mathbf u=\sum_{i=1}^s\alpha_i \mathbf v_i + \mathbf w$, for some $\alpha_1,\ldots,\alpha_s\in \mathbb N$, with $\sum_{i=1}^s\alpha_i = a$. If $\#\mathsf Z(\mathbf a)>1$, then $a>0$. 

\end{itemize} 
\end{proposition}
\begin{proof}
(a) Assume that $\mathbf{a}=\mathbf{b}+\mathbf d$ for some $\mathbf b\in S$. Take $\mathbf w\in \mathsf Z(\mathbf b)$. Since $\mathbf d\in\mathrm{Betti}(S)$, there exists $\mathbf u\neq \mathbf v\in \mathsf Z(\mathbf d)$. Hence $\mathbf w+ \mathbf u\neq \mathbf w+\mathbf v\in \mathsf Z(\mathbf a)$. Conversely, if $\#\mathsf Z(\mathbf a)>1$, then \cite[Lemma 1]{uniquely}, forces $\mathbf a -\mathbf d\in S$.

\noindent (b) Suppose on the contrary that there exists $j\neq k\in \{1,\ldots,s\}$ and  $i \in \{1, \ldots, p\}$ with $i \in \mathrm{supp}(\mathbf{v}_j) \cap \mathrm{supp}(\mathbf{v}_k).$ Let $\mathbf{a} = \pi_\mathcal{A}(\mathbf{v}_j - \mathbf{e}_i).$ Then $\mathbf{a}-\mathbf{d}\not\in S$ and $\#\mathsf Z(\mathbf{a}) \geq 2$ (because $\mathbf{v}_j - \mathbf{e}_i\neq \mathbf{v}_k - \mathbf{e}_i \in \mathbb{N}^p$), in contradiction with the preceding statement.

\noindent (c) If $\pi^{-1}(\mathbf{a}) = \{\mathbf{u}\}$, we are done by taking $\mathbf{b} = \mathbf{a}$ and $\alpha_1 = \cdots = \alpha_s = 0$. Otherwise, from the proof of \cite[Lemma 1]{uniquely}, there exists $\mathbf{v} \in \mathsf{Z}(\mathbf{d}),$ such that $\mathbf{u} - \mathbf{v} = \mathbf{u}' \in \mathbb{N}^p$. Let $\mathbf{a}' = \pi_\mathcal{A}(\mathbf{u}');$ thus we have $\mathbf{a}' \in S$ and $\mathbf{u}' \in \mathsf{Z}(\mathbf{a}').$  So by repeating the above argument as many times as needed we obtain $\mathbf{u} = \sum_{i=1}^s \alpha_i \mathbf{v}_i + \mathbf{w}$ with $\pi_\mathcal{A}(\mathbf{w}) = \mathbf{b},$ for some $\alpha_i \in \mathbf{N},\ i = 1, \ldots, k,$ and $\mathbf{b} \in S$ with unique factorization. This process stops since $\mathbf u'< \mathbf u$ in $\mathbb N^p$ and there are no infinite descending chains with respect to $<$. 

Let us see that $a$ and $\mathbf b$ (and thus $\mathbf w$), depend exclusively on $\mathbf a$. Notice that $\mathbf{a} = \pi_\mathcal A(\mathbf{u}) = \pi_\mathcal A\big(\sum_{i=1}^s \alpha_i \mathbf{v}_i + \mathbf{w}\big) =  \pi_\mathcal A\big(\sum_{i=1}^s \alpha_i \mathbf{v}_i \big) + \pi_\mathcal A(\mathbf{w}) = a \mathbf{d} + \mathbf{b}$. If $a\mathbf d+\mathbf b= a'\mathbf d+\mathbf b'$ with $a>a'\in \mathbb N$ and $\mathbf b,\mathbf b'\in S$ with unique factorizations, then $(a-a')\mathbf d+\mathbf b=\mathbf b'$. This leads to $\mathbf b'-\mathbf d\in S$, which contradicts the fact that $\mathbf b'$ has a unique factorization (1).
\end{proof}

\subsection*{Gluings}

Let $\mathcal{A}_1$ and $\mathcal{A}_2$ be disjoint nontrivial subsets of $\mathcal{A}$ such that $\mathcal{A} = \mathcal{A}_1 \cup \mathcal{A}_2$. Set $S_1$ and $S_2$ to be the affine semigroups generated by $\mathcal A_1$ and $\mathcal A_2,$ respectively. We say that $S$ is the \textbf{gluing} of $S_1$ and $S_2$ if there exists $\mathbf{d} \in S_1\cap S_2\setminus\{0\}$ such that $\mathbb{Z} \mathcal{A}_1 \cap \mathbb{Z} \mathcal{A}_2 = \mathbb{Z} \mathbf{d}$, where $\mathbb{Z} \mathcal{A}$ stands for the subgroup of $\mathbb Z^p$ generated by $\mathcal{A}$ (and $\mathbb{Z}$ is the set of integers). We also say that $S$ is the gluing of $S_1$ and $S_2$ by $\mathbf{d}$.

\begin{proposition}\label{Th1.4.Ro97}
If $S$ is the gluing of $S_1$ and $S_2$, then $\nu(S) = \nu(S_1) + \nu(S_2) + 1.$
\end{proposition}

\begin{proof}
See \cite[Theorem 1.4]{Ro97}.
\end{proof}

As an easy corollary of this result, we obtain the following (see \cite[Theorem 3.1]{fischer97}).

\begin{corollary}\label{Cor1.4.Ro97}
Let $S$ be the gluing of $S_1$ and $S_2.$ If $S_1$ and $S_2$ are complete intersections, then $S$ is a complete intersection.
\end{corollary}

It is also known what are the Betti elements  of a gluing.

\begin{proposition}\label{Th10.uniquely}
If $S$ is the gluing of $S_1$ and $S_2$ by $\mathbf{d}$, then $\mathrm{Betti}(S) = \mathrm{Betti}(S_1) \cup \mathrm{Betti}(S_2) \cup \{\mathbf{d}\}.$
\end{proposition}

\begin{proof}
See \cite[Theorem 10]{uniquely}.
\end{proof}

\begin{example}
Let $S$ be the subsemigroup of $\mathbb{N}^2$ generated by the columns of $$A = \left(\begin{array}{ccc} 2 & 0 & 1 \\ 0 & 2 & 1 \end{array}\right).$$ Clearly, $S$ is the gluing of $S_1 = \mathbb{N} \left(\begin{array}{c} 2 \\ 0 \end{array}\right) + \mathbb{N} \left(\begin{array}{c} 0 \\ 2 \end{array}\right)$ and $S_2 = \mathbb{N} \left(\begin{array}{c} 1 \\ 1 \end{array}\right)$ by $\mathbf{d} = \left(\begin{array}{c} 2 \\ 2 \end{array}\right).$ Since $\mathrm{Betti}(S_1) = \mathrm{Betti}(S_2) = \varnothing$ (because both semigroups are free) we obtain that $\mathrm{Betti}(S) = \{\mathbf{d}\}$ by Proposition \ref{Th10.uniquely}. Now, by Proposition \ref{Th1.4.Ro97}, any minimal presentation of $S$ has cardinality $1.$ So, we conclude that $\sigma =  \{\big( (1,1,0), (0,0,2) \big) \}$ is a minimal presentation of $S.$
\end{example}

\subsection*{Primitive elements}

Let $\mathcal I(\mathcal{A})$ be the set of irreducible elements (atoms) of $\equiv_\mathcal{A}$, that is, the set of minimal elements of $\equiv_\mathcal{A}\setminus\{(0,0)\}$ with respect the usual partial order on $\mathbb N^p\times \mathbb N^p$. The set $\mathcal I(\mathcal{A})$ generates $\equiv_\mathcal{A}$ as a monoid (\cite[Chapter 8]{RGS99}). Denote by $\mathrm{Gr}_\mathcal{A}=\mathcal I (\mathcal{A})\setminus\{(\mathbf e_1,\mathbf e_1),\ldots, (\mathbf e_d,\mathbf e_d)\}$, which is known as the set of \textbf{primitive} elements of $\equiv_\mathcal{A}$ or the \textbf{Graver basis} of $\mathcal{A}$. This set generates $\equiv_\mathcal{A}$ as a congruence, it actually contains every minimal presentation of $\equiv_\mathcal{A}$ (see, e.g. \cite[Chapter 4]{Sturmfels}).

\subsection*{Circuits}

A nonzero element $(\mathbf{u}, \mathbf{v})$ in the kernel congruence of $\pi_\mathcal{A}$ is called a \textbf{circuit} of $\equiv_\mathcal{A}$ if $\mathbf{u}$ and $\mathbf{v}$ have disjoint support, $\mathrm{supp}(\mathbf{u}) \cup \mathrm{supp}(\mathbf{v})$ is minimal with respect to inclusion, and the coordinates of $(\mathbf{u}, \mathbf{v})$ are relative prime. The set of circuits of $\equiv_\mathcal{A}$ is denoted by $\mathcal{C}_\mathcal{A}.$ It is worth to be noted that the circuits of $\equiv_\mathcal{A}$ can be computed by using elementary linear algebra (see the proof of \cite[Lemma 4.9]{Sturmfels}).

\medskip
Now we are in a position to state our main theorem:

\begin{theorem}\label{Th.Main}
Let $S$ be an affine semigroup minimally generated by $\mathcal{A} = \{\mathbf{a}_1,\ldots, \mathbf{a}_p\} \subset \mathbb{N}^r$, and let $\mathbf{d} \in S$ be nonzero. The following are equivalent:
\begin{itemize}
\item[(a)] $\mathrm{Betti}(S) = \{ \mathbf{d} \}$.
\item[(b)] $\mathcal{C}_\mathcal{A} = \mathsf{Z}(\mathbf{d}) \times \mathsf{Z}(\mathbf{d}) \setminus \{ (\mathbf{u}, \mathbf{u}) \mid \pi_\mathcal{A}(\mathbf{u}) = \mathbf{d} \} = \mathrm{Gr}_\mathcal{A}$. 
\item[(c)] $S$ is gluing of some $S_1$ and $S_2$ by $\mathbf{d},$ with $\mathrm{Betti}(S_1) \cup \mathrm{Betti}(S_2) \subseteq \{\mathbf{d}\}$.
\end{itemize}
\end{theorem}

\begin{remark}\label{soportes disjuntos}
Notice that as a consequence of second equality in (b), we also obtain that $\mathbf{u} \cdot \mathbf{v}=0,$ for every $\mathbf u, \mathbf v\in \mathsf Z(\mathbf d), \mathbf u\neq \mathbf v$  (since otherwise $(\mathbf u,\mathbf v)$ would not be a circuit).
\end{remark}

\begin{proof}
\noindent \framebox{(a) $\Rightarrow$ (b).} We will prove \[\mathcal{C}_\mathcal{A} \subseteq \mathrm{Gr}_\mathcal{A} \subseteq \mathsf{Z}(\mathbf{d}) \times \mathsf{Z}(\mathbf{d}) \setminus \{ (\mathbf{u}, \mathbf{u}) \mid \pi_\mathcal{A}(\mathbf{u})  = \mathbf{d} \} \subseteq \mathcal C_\mathcal A.\]

The inclusion $\mathcal{C}_\mathcal{A} \subseteq \mathrm{Gr}_\mathcal{A}$ is well known, see \cite[Proposition 4.11]{Sturmfels}. 

Now take $(\mathbf u,\mathbf v)\in \mathrm{Gr}_\mathcal{A}$. By Proposition \ref{prop-unificadora} (c), applied to $\mathbf u$ and $\mathbf v$, there exists $\mathbf u',\mathbf v'\in \mathsf Z(\mathbf{d})$ such that $(\mathbf u,\mathbf v)-(\mathbf u',\mathbf v')\in \mathbb N^p\times \mathbb N^p$. The minimality of a primitive element with respect to the usual partial order, forces $(\mathbf u,\mathbf v)=(\mathbf u',\mathbf v')$. This proves that $\mathrm{Gr}_\mathcal{A}\subseteq \mathsf{Z}(\mathbf{d}) \times \mathsf{Z}(\mathbf{d}) \setminus \{ (\mathbf{u}, \mathbf{u}) \mid \pi_\mathcal{A}(\mathbf{u})  = \mathbf{d} \}$.

Let $(\mathbf{u}, \mathbf{v}) \in  \mathsf{Z}(\mathbf{d}) \times \mathsf{Z}(\mathbf{d}),\ \mathbf{u} \neq \mathbf{v}$. By the minimality of the supports of the elements of $\mathcal C_\mathcal A$, there exists $(\mathbf{u}', \mathbf{v}') \in \mathcal{C}_\mathcal{A}$ such that $\mathrm{supp}(\mathbf{u}')\cup \mathrm{supp}(\mathbf{v}')  \subseteq \mathrm{supp}(\mathbf{u}) \cup \mathrm{supp}(\mathbf{v})$. Since we already know that $\mathbf{u}', \mathbf{v}' \in \mathsf{Z}(\mathbf{d})$, and Propositon \ref{prop-unificadora} asserts that any two different factorizations of $\mathbf d$ have disjoint support, we conclude that either $\mathbf u'=\mathbf u$ and $\mathbf v'=\mathbf v$, or $\mathbf v'=\mathbf u$ and $\mathbf u'=\mathbf v$. Notice that $(\mathbf v', \mathbf u')$ is again a circuit, and this shows that $\mathsf{Z}(\mathbf{d}) \times \mathsf{Z}(\mathbf{d}) \setminus \{ (\mathbf{u}, \mathbf{u}) \mid \pi_\mathcal{A}(\mathbf{u})  = \mathbf{d} \} \subseteq \mathcal C_\mathcal A$.

\noindent \framebox{(b) $\Rightarrow$ (c).}
Let $(\mathbf u,\mathbf v)\in \mathrm{Gr}_\mathcal A$. Clearly, $\mathrm{Gr}_\mathcal A \setminus\{(\mathbf v,\mathbf u)\}$ is also a system of generators of $\equiv_\mathcal A$ as a congruence. Let $\mathcal A_1=\{\mathbf a_i \ \mid\ i\in \mathrm{supp}(\mathbf u)\}$, and $\mathcal A_2=\mathcal A\setminus \mathcal A_1$. Set $S_1=\langle \mathcal A_1\rangle$ and $S_2=\langle\mathcal A_2\rangle$. Then $\mathrm{Gr}_\mathcal A\setminus\{(\mathbf v,\mathbf u)\}$ is of the form $\sigma\cup\{(\mathbf u,\mathbf v)\}$ with no pair in $\sigma$ having common support with $\mathbf u$ (due to Remark \ref{soportes disjuntos}). Hence \cite[Theorem 1.4]{Ro97}, implies that $S$ is the gluing of the (free) semigroup $S_1$ and $S_2$.

\noindent \framebox{(c) $\Rightarrow$ (a)} It follows easily from Proposition \ref{Th10.uniquely}. 
\end{proof}

\begin{example}\label{example_K}
Let $A$ be the incidence matrix of the complete graph K4, that is to say, $$A = \left(\begin{array}{cccccc} 1 & 1 & 1 & 0 & 0 & 0 \\ 1 & 0 & 0 & 1 & 1 & 0 \\ 0 & 1 & 0 & 1 & 0 & 1 \\ 0 & 0 & 1 & 0 & 1 & 1 \end{array}\right)$$ and let $S$ be the subsemigroup of $\mathbb{N}^ 4$ generated by the columns of $A.$ It is easy to see that $S$ is the gluing of $S_1, S_2$ and $S_3$ by $\mathbf{d} = (1,1,1,1),$ where $S_1$ is the free semigroup generated by the first and sixth columns, $S_2$ is the free semigroup generated by the second and fifth columns and $S_3$  is the free semigroup generated by the third and fourth columns. So, by Theorem \ref{Th.Main}, $S$ has a unique Betti element, in this case, $\mathrm{Betti}(S) = \{ (1,1,1,1) \}.$
\end{example}

There are important families of affine semigroups that satisfies the equality $\mathcal{C}_\mathcal{A} = \mathrm{Gr}_\mathcal{A};$ for instance, affine semigroups generated by the columns of $r \times p-$unimodular matrices have this property (see, \cite[Proposition 8.11]{Sturmfels}). Nevertheless, the condition $\mathcal{C}_\mathcal{A} = \mathrm{Gr}_\mathcal{A}$ is not sufficient for a semigroup to have a unique Betti element as the following example shows:

\begin{example}
The incidence matrix $A$ of the graph:

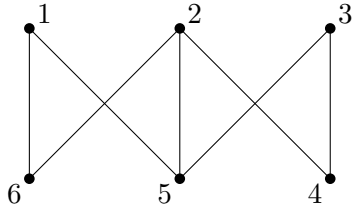
\begin{tikzpicture}[x=2.0cm,y=2.0cm]
\clip(-3,-0.5) rectangle (5,1.5);
\draw (1,0)-- (0,1);
\draw (0,0)-- (1,1);
\draw (2,0)-- (1,1);
\draw (1,0)-- (2,1);
\draw (0,1)-- (0,0);
\draw (1,1)-- (1,0);
\draw (2,1)-- (2,0);
\fill  (0,1) circle (2.0pt);
\draw (0.1,1.1) node {$1$};
\fill  (1,1) circle (2.0pt);
\draw (1.1,1.1) node {$2$};
\fill  (2,1) circle (2.0pt);
\draw (2.1,1.1) node {$3$};
\fill (0,0) circle (2.0pt);
\draw (-0.1,-0.1) node {$6$};
\fill  (1,0) circle (2.0pt);
\draw (0.9,-0.1) node {$5$};
\fill  (2,0) circle (2.0pt);
\draw (1.9,-0.1) node {$4$};
\end{tikzpicture}

\noindent is a $6 \times 7$-unimodular matrix. If $S$ is the subsemigroup of $\mathbb{N}^6$ generated by the columns of $A,$ then $\mathrm{Betti}(S) = \{ (1,1,0,0,1,1), (0,1,1,1,1,0)\}$ (see \cite{KaTh} for the details).
\end{example}

From the results in \cite{Ka11}, it follows that the only affine semigroup defined by the incidence matrix of a graph having a unique Betti element is the semigroup associated to K4 (see Example \ref{example_K}).

\begin{corollary}
Let $\mathbf{d} \in S$ be nonzero. If $\mathrm{Betti}(S) = \{ \mathbf{d} \},$ then 
\begin{itemize}
\item[(a)] $\mathcal{C}_\mathcal{A}$ is a presentation of $S.$
\item[(b)] $S$ is the gluing of $\# \mathsf{Z}(\mathbf{d})$ free affine semigroups.
\item[(c)] $S$ is a complete intersection.
\end{itemize}
\end{corollary}

\begin{proof}
(a) It follows directly from the definition of $\mathrm{Betti}(S)$ and the implication (a) $\Rightarrow$ (b) in Theorem \ref{Th.Main}.

\noindent (b) By using recursively the implication (b) $\Rightarrow$ (c) in Theorem \ref{Th.Main}, we obtain that $S$ is a gluing of free semigroups. More precisely, if $\mathsf{Z}(\mathbf{d}) = \{\mathbf{u}_1, \ldots, \mathbf{u}_n\} \subset \mathbb{N}^r,$ then $S$ is the gluing of $S_1, \ldots, S_n,$ where $S_i$ is the semigroup generated by $\{ \mathbf{a}_j \mid j \in \mathrm{supp}(\mathbf{u}_i) \},\ i = 1, \ldots, n-1,$ and $S_n$ the semigroup generated by $\{ \mathbf{a}_j \mid j \in \mathrm{supp}(\mathbf{u}_n) \}\cup \{ \mathbf{a}_k \mid k \not\in \mathrm{supp}(\mathbf{u}_i),\ i = 1, \ldots, n\}$.

\noindent (c) Our claim follows from part (b), by Corollary \ref{Cor1.4.Ro97}.
\end{proof}

\begin{remark}
 From Proposition \ref{prop-unificadora} (b) and \cite[Chapter 9]{RGS99} it also follows that if $\mathsf Z(\mathbf d)=\{\mathbf v_1,\ldots, \mathbf v_s\}$, then every minimal presentation of $S$ is of the form $\{ (\mathbf v_i,\mathbf v_j) ~|~  (i,j)\in I\}$, where $I$ is the set of edges of any spanning tree of the complete graph on the vertices $\{1,\ldots, d\}$.
\end{remark}

\begin{example}\label{exNS}\textbf{Numerical semigroups with a single Betti element.}
A numerical semigroup is a submonoid of $\mathbb N$ with finite complement in $\mathbb N$. It follows that $\gcd(S)=1$. Any submonoid $M$ of $\mathbb N$ is isomorphic to a numerical semigroup, since it suffices to divide its elements by the greatest common divisor of $M$ (see \cite[Proposition 2.2]{libro-ns}). Every numerical semigroup has a unique minimal system of generators, which is always finite (see for instance \cite[Theorem 2.7]{libro-ns}). Thus any numerical semigroup is an affine semigroup.

Let $S$ be a numerical semigroup minimally generated by $\{n_1,\ldots,n_p\}$, $p\ge 2$. Set 
\[c_i=\min\{ k\in \mathbb N\setminus\{0\}~|~ k n_i \in \langle n_1,\ldots, n_{i-1},n_{i+1},\ldots,n_p \rangle \}.\]

Notice that \[\mathcal C_\mathcal A=\left\{\left(\frac{n_i}{\gcd(n_i,n_j)} \mathbf e_j, \frac{ n_j}{\gcd(n_i,n_j)} \mathbf e_i\right)~\big\vert~ i\neq j\right\}.\] 
Moreover if $c_j n_j=\sum_{i\neq j} r_{j_i}n_i$ then the minimality of $c_i$ implies $(c_j \mathbf e_j,  \sum_{i\neq j} r_{j_i}\mathbf e_i)\in \mathrm{Gr}_\mathcal A$ and $c_i \leq \frac{n_j}{\gcd(n_i,n_j)}$ for every $j \neq i.$

The following conditions are equivalent.
\begin{itemize}
\item[1.] The set $\mathrm{Betti}(S)$ is a singleton.
\item[2.] There exists $k_1, \ldots, k_p$ pairwise relatively prime integers greater than one, such that $n_i=\prod_{j\neq i} k_j$ for all $i\in\{1,\ldots,p\}$.
\end{itemize}
Furthermore, if this is the case, $k_i = c_i,\ i = 1, \ldots, p$ and $\mathrm{Betti}(S)=\{c_1 a_1\}=\{c_in_i ~|~ i\in\{1,\ldots, p\}\}$.

\noindent \framebox{1 $\Rightarrow$ 2.} Let $\mathrm{Betti}(S)=\{d\}$. By Theorem \ref{Th.Main} (a) $\Rightarrow$ (b), we deduce that  $c_1 n_1=\cdots = c_p n_p =d.$ Moreover, since $\mathrm{gcd}(c_i, c_j) = 1$ and $c_i$ divides $n_j$ for every $i \neq j,$ we conclude that $n_j = \prod_{i \neq j} c_j.$

\noindent \framebox{2 $\Rightarrow$ 1.} We use induction on $p$. The case $p=2$ follows trivially. Observe that $k_1=\gcd\{n_2,\ldots,n_p\}$. If we set $S_1=\langle n_1\rangle$ and $S_2=\langle n_2,\ldots,n_p\rangle$, then $\mathrm G(S_1)\cap \mathrm G(S_2)=\mathbb{Z} n_1 \cap \mathbb{Z} k_1 = \mathbb Z k_1n_1$. We obtain that $S$ is the gluing of the monoids $S_1$ and $S_2$.  By induction hypothesis,  $S'=\langle \frac{n_2}{k_1},\ldots,\frac{n_p}{k_1}\rangle$ has a unique Betti element (observe that if $d'$ is the only Betti element of $S'$, then $\mathrm{Betti}(S_2)=\{k_1d'\}$). Define $c_i'=\min\{k\in \mathbb N\setminus\{0\}~|~ k\frac{n_i}{k_1}\in \langle \{\frac{n_2}{k_1},\ldots,\frac{n_p}{k_{i-1}} ,\frac{n_p}{k_{i+1}}, \ldots,\frac{n_p}{k_p}\rangle\}$. By Theorem \ref{Th.Main} (a) $\Rightarrow$ (b), we obtain $c_2'\frac{n_2}{k_1}=\cdots = c_p'\frac{n_p}{k_1}$. Hence $c_2'n_2=\cdots= c_p'n_p$. Notice also that $k_2n_2=\cdots=k_pn_p$. By the definition of $c_2'$,  $c_2'\le k_2$. Besides, $c_3'=c_2'\frac{k_3}{k_2}$, and since $\gcd\{k_2,k_3\}=1$, we deduce that $k_2$ divides $c_2'$. Hence $c_2'=k_2$. In the same way, one proves that $c_i'=k_i$ for $i\in\{3,\ldots,p\}$. Now by Proposition \ref{Th10.uniquely}, we have that $\mathrm{Betti}(S)=\mathrm{Betti}(S_1)\cup \mathrm{Betti}(S_2)\cup \{k_1n_1\}=\varnothing\cup \{k_1 k_2\frac{n_2}{k_1}\}\cup\{ k_1 n_1\}=\{k_1 n_1\}$. 

In the context of toric ideals, a proof of the implication 2 $\Rightarrow$ 1 can be found in \cite{KaOj}.

As as consequence of the equivalence 1 $\Leftrightarrow$ 2, we can construct as many examples of numerical having just one Betti element as we like: it suffices to take $p$ relatively prime positive integers, and then take the numerical semigroup (minimally) generated by the products of all $p-1$ subsets with $p-1$ elements. For instance, if $c_1 = 7, c_2 = 5, c_3 = 3$ and $c_4 = 2,$ the numerical semigroup $S$ generated by $n_1 = c_2 c_3 c_4 = 30, n_2 = c_1 c_3 c_4 = 42, n_3 = c_1 c_2 c_4 = 70$ and $n_4 = c_1 c_2 c_3 = 105$ has only one Betti element, $d = 210.$ A minimal presentation for $S$ is $\big\{ (7 \mathbf{e}_1, 5 \mathbf{e}_2), 
(7 \mathbf{e}_1, 3 \mathbf{e}_3), (7 \mathbf{e}_1, 2 \mathbf{e}_2) \big\}.$
\end{example}

\section{Non-unique factorization invariants}

We next compute some non-unique factorization invariants for affine semigroups with a single Betti element. We recall their definitions here, though in same cases we present a simplified version that fits in our scope. The interested reader in these invariants is referred to \cite{GHKb}.

Throughout this section, $S$ will denote an affine semigroup minimally generated by $\mathcal{A} = \{\mathbf{a}_1, \ldots, \mathbf{a}_p\} \subset \mathbb{N}^r.$

Given $\mathbf{a} \in S$ and $\mathbf{u} = (u_1, \ldots, u_p)  \in \mathsf{Z}(\mathbf{a}),$ set $\vert \mathbf{u} \vert = u_1 + \cdots + u_p$ and define the \textbf{set of lengths} of $\mathbf{a}$ as $\mathsf{L}(\mathbf{a}) = \big\{ \vert \mathbf{u} \vert : \mathbf{u} \in  \mathsf{Z}(\mathbf{a}) \big\}.$

\begin{lemma}\label{lemma-key-fact}
Let $\mathbf{a} \in S$. If $\mathrm{Betti}(S) = \{\mathbf{d}\},$ then there exist $a$ and $b \in \mathbb{N}$ such that for every $\mathbf{u} \in \mathsf{Z}(\mathbf{a}),$  $$ a \min(\mathsf L (\mathbf{d})) + b \leq \vert \mathbf{u} \vert \leq a \max(\mathsf L (\mathbf{d})) + b.$$
\end{lemma}

\begin{proof}
By Proposition \ref{prop-unificadora} (c), $\mathbf{u} = \sum_{i=1}^s \alpha_i \mathbf{v}_i + \mathbf{w}$ for some $\alpha_i \in \mathbb{N}$ and $\mathbf{w} \in \mathbb{N}^p.$ So, it suffices to take $a = \sum_{i=1}^s \alpha_i$ and $b = \vert \mathbf{w} \vert$ (recall that $a$ and  $\mathbf w$ depend exclusively on $\mathbf a$). 
\end{proof}

Notice that, with the above notation, if $j \not\in \mathrm{supp}(\mathbf{w})$ and $\mathbf{a} - \mathbf{a}_j \in S,$ there exists $\mathbf{v}_k$ such that $i \in \mathrm{supp}(\mathbf{v}_k).$ Otherwise, $\mathbf{b} = \pi_\mathcal{A}(\mathbf{w})$ will have more than one factorization.

\begin{definition}
The \textbf{elasticity} of $\mathbf{a} \in S$ is defined as $\rho(\mathbf{a})=\frac{\max(\mathsf L(\mathbf{a}))}{\min(\mathsf L(\mathbf{a}))}$, and the  \textbf{elasticity} of $S$ is \[\rho(S)=\sup\left\{\rho(\mathbf{a}) ~|~ s\in S\setminus\{0\}\right\}.\]
\end{definition}

The following result can be deduced from \cite[Theorem 15]{atomicos} and Theorem \ref{Th.Main}, we include an alternative (and straightforward) proof for the single Betti element setting.

\begin{corollary}
If $\mathrm{Betti}(S) = \{\mathbf{d}\},$ then $\rho(S) = \frac{\max(\mathsf L(\mathbf{d}))}{\min(\mathsf L(\mathbf{d}))}.$
\end{corollary}

\begin{proof}
Let $\mathbf{a} \in S.$ By Lemma \ref{lemma-key-fact}, there exist $a$ and $b \in \mathbb{N}$ such that $\max(\mathsf L(\mathbf{a})) \leq a\, \max(\mathsf L(\mathbf{d})) + b$ and  $\min(\mathsf L(\mathbf{a})) \geq a\, \min(\mathsf L(\mathbf{d})) + b.$ Therefore, $$\rho(\mathbf{a})=\frac{\max(\mathsf L(\mathbf{a}))}{\min(\mathsf L(\mathbf{a}))} \leq \frac{a\, \max(\mathsf L(\mathbf{d})) + b}{a\, \min(\mathsf L(\mathbf{d})) + b} \leq \frac{\max(\mathsf L(\mathbf{d}))}{\min(\mathsf L(\mathbf{d}))}$$ and the claim follows.
\end{proof}

Let $\mathbf{a} \in S.$ Since $\mathsf{L}(\mathbf{a})$ is finite, we can write $\mathsf{L}(\mathbf{a}) = \{n_1 < \cdots < n_s\}.$ With this notation, the \textbf{delta set} of $\mathbf{a}$ is $\Delta(\mathbf{a}) := \big\{ n_i - n_{i-1} \mid 2 \leq i \leq s\}$ and the \textbf{delta set} of $S$ is the union of the delta sets of the elements in $S,$ in symbols, $$\Delta(S) := \bigcup_{\mathbf{a} \in S \setminus \{0\}} \Delta(\mathbf{a}).$$

It is known that the maximum of the delta set of $S$ is reached in the Betti elements of $S$ (see \cite[Theorem 2.5]{delta}); therefore, if $\mathrm{Betti}(S) = \{\mathbf{d}\},$ then $$\max(\Delta(S)) = \max(\Delta(\mathbf{d})).$$

\medskip
For $\mathbf{u} = (u_1, \ldots, u_p)$ and $\mathbf{v} = (v_1, \ldots, v_p) \in \mathbb{N}^p,$ we define $\mathbf{u}\wedge  \mathbf{v} = \big( \min(u_1,v_1), \ldots, \min(u_p,v_p) \big).$ Notice that $\mathbf{u} - \mathbf{u}\wedge \mathbf{v}$ and $\mathbf{v} - \mathbf{u}\wedge \mathbf{v}$ lie in $\mathbb{N}^p,$ and this is the largest element in $\mathbb{N}^p$ fulfilling this condition with respect to the usual partial ordering on $\mathbb{N}^p.$

Let $\mathsf{d} : \mathbb{N}^p \times \mathbb{N}^p \to \mathbb{N}$ be such that $$\mathsf{d}(\mathbf{u}, \mathbf{v}) = \max \big\{ \vert \mathbf{u} - \mathbf{u}\wedge \mathbf{v} \vert,  \vert \mathbf{v} - \mathbf{u}\wedge \mathbf{v} \vert \big\}.$$ It is not difficult to see that $\mathsf{d}$ is actually a metric in the topological sense (see \cite[Proposition 1.2.5]{GHKb}). 

\begin{definition}\label{Def.catenary}
The \textbf{catenary degree} of $S, \mathsf{c}(S),$ is the minimum $N \in \mathbb{N}$ such that for any $s \in S$ and for any $\mathbf{u}, \mathbf{v} \in \mathsf{Z}(\mathbf{a})$ there exists a sequence $\mathbf{u}_0, \ldots, \mathbf{u}_k \in \mathsf{Z}(\mathbf{a})$ such that 
\begin{itemize}
\item[(a)] $\mathbf{u}_0 = \mathbf{u}$ and $\mathbf{u}_k = \mathbf{v}$, 
\item[(b)] $\mathsf{d}(\mathbf{u}_{i-1}, \mathbf{u}_i) \leq N,\ i \in\{ 1, \ldots, k\}$.
\end{itemize}
\end{definition}

\begin{definition}
The \textbf{$\omega$-primality} of $S$, $\omega(S)$, is the minimum of all $N\in \mathbb N\cup\{\infty\}$ such that for every minimal generator $\mathbf{a}_j$, if $\sum_{i\in I} \mathbf{b}_i-\mathbf{a}_j\in S$, with $\mathbf{b}_i\in S$ for all $i\in I$, then there exists $\Omega\subseteq I$ with $\# \Omega \leq N$ such that $\sum_{i\in \Omega} \mathbf{b}_i-\mathbf{a}_j\in S$.
\end{definition}

\begin{definition}
The \textbf{tame degree} of $S$, $\mathsf t(S)$, is the minimum of all $N\in \mathbb N\cup\{\infty\}$ such that for all $\mathbf{a} \in S$, $\mathbf{u}\in {\mathsf Z}(\mathbf{a})$ and every minimal generator $\mathbf{a}_i$ such that $\mathbf{a}-\mathbf{a}_i\in S$, there exists $\mathbf{u}' \in {\mathsf Z}(\mathbf{a})$ such that $i \in \mathrm{supp}(\mathbf{u}')$ and $\mathsf{d}(\mathbf{u},\mathbf{u}')\leq N$.
\end{definition}

Next we prove that for affine semigroups with a single Betti element, the catenary and tame degrees coincide with the $\omega$-primality as occurs in the generic case (\cite{BGSG}). We can also give an explicit value for these invariants.

\begin{theorem}\label{Th.Main2}
Let $S$ be an affine semigroup minimally generated by $\mathcal{A} = \{\mathbf{a}_1,\ldots, \mathbf{a}_p\} \subset \mathbb{N}^r.$ If $\mathrm{Betti}(S) = \{\mathbf{d}\},$ then $$\mathsf{c}(S) = \omega(S) = \mathsf{t}(S)=  \max(\mathsf{L}(\mathbf{d})).$$
\end{theorem}

\begin{proof}
By \cite[Section 3]{Ge-Ka10a}, $\mathsf c(S)\le \omega(S)$, and from \cite[Theorem 3.6]{Ge-Ha08a}, $\omega(S)\le \mathsf t(S)$. Hence 

\begin{equation}\label{desigualdades}
\mathsf{c}(S) \leq \omega(S) \leq \mathsf{t}(S).
\end{equation}

Let $\mathsf Z(\mathbf{d}) = \{ \mathbf{v}_1, \ldots, \mathbf{v}_s\}.$ 
By Proposition \ref{prop-unificadora} (c), $\mathbf{u}_i\wedge \mathbf{u}_j = 0,\ i \neq j.$ Therefore, given $\mathbf{u}, \mathbf{v} \in \mathsf Z(\mathbf{d})$ such that $\vert \mathbf{u} \vert = \max(\mathsf{L}(\mathbf{d})),$ we have that $\mathsf d(\mathbf{u}_0, \mathbf{u}_1) = \max(\mathsf{L}(\mathbf{d}))$ for every sequence $\mathbf{u}_0, \ldots, \mathbf{u}_s \in \mathsf{Z}(\mathbf{a})$ satisfying (a) and (b) in Definition \ref{Def.catenary}. Therefore, $$\max(\mathsf{L}(\mathbf{d})) \leq \mathsf{c}(S).$$

Finally, for all $\mathbf{a} \in S \setminus \{0\}$, $\mathbf{u}\in {\mathsf Z}(\mathbf{a})$ and every minimal generator $\mathbf{a}_i$ such that $\mathbf{a}-\mathbf{a}_i\in S$, if $i \not\in \mathrm{supp}(\mathbf{u}),$ there exists $\mathbf{u}' \in {\mathsf Z}(\mathbf{a})$ such that $i \in \mathrm{supp}(\mathbf{u}')$ and $\mathbf{u} - \mathbf{u}\wedge \mathbf{u}', \mathbf{u}' - \mathbf{u}\wedge \mathbf{u}' \in \mathsf Z(\mathbf{d}).$ Indeed, by Proposition \ref{prop-unificadora}, $\mathbf{u}  = \sum_{j=1}^s \alpha_j \mathbf{v}_j + \mathbf{w},$ for some $\alpha_j \in \mathbb{N}$ and $\mathbf{w} \in \mathbb{N}^p.$ Since $i \not\in \mathrm{supp}(\mathbf{w}),$ there exists $k$ such that $i \in \mathrm{supp}(\mathbf{v}_k)$ (in particular, $\alpha_k = 0$). Now, if $l \in \{1, \ldots, s\}$ is such that $\alpha_l \neq 0,$ by taking $\mathbf{u}' = \sum_{j=1}^s \alpha_j \mathbf{v}_j + \mathbf{v}_k - \mathbf{v}_l + \mathbf{w}$ we are done.

So, either $\mathsf d(\mathbf{u}, \mathbf{u}') = 0$ or $\mathsf d(\mathbf{u}, \mathbf{u}') \le \max(\mathsf{L}(\mathbf{d}))$,  and consequenly the minimality of $\mathsf t(S)$, yields $$\mathsf{t}(S) \leq \max(\mathsf{L}(\mathbf{d})).$$
The proof now follows from (\ref{desigualdades}).
\end{proof}

In \cite{catenary} it is shown (though with a different notation) that the catenary degree of $S$ depends only on the catenary degree of the Betti elements of $S$, and that the tame degree can be computed from the factorizations of the elements involved in $\mathrm{Gr}_\mathcal A$. In view of Theorem \ref{Th.Main}, if $S$ has a single Betti element, then the only element whose factorizations appear in $\mathrm{Gr}_\mathcal A$ is precisely this single Betti element. By using this fact one could have give an alternative proof to Theorem \ref{Th.Main}.

\end{document}